\documentclass[11pt]{article}

\usepackage{amsmath,euscript}
\usepackage{amssymb}
\usepackage{amsthm}
\usepackage{enumerate}
\usepackage{amsfonts}
\usepackage{comment}
\usepackage{colortbl}
\usepackage{amssymb,amsmath,array}

\usepackage{latexsym}

\usepackage{amscd}

\input{xy}
\xyoption{all}

\newcommand{\gc}{ [ \hspace{-0.65mm} [}
\newcommand{\dc}{]  \hspace{-0.65mm} ]}

\newcommand{\oc}{\mathcal{O}}
\newcommand{\mc}{\mathcal{M}}

\newcommand{\ia}{i,\alpha}
\newcommand{\hc}{\mathcal{H}}

\newtheorem{theo}{Theorem}[section]
\newtheorem{prop}[theo]{Proposition}
\newtheorem{lem}[theo]{Lemma}
\newtheorem{cor}[theo]{Corollary}

\theoremstyle{definition}
\newtheorem{defi}[theo]{Definition}
\newtheorem{nota}[theo]{Notation}
\newtheorem{conv}[theo]{Convention}

\newcommand{\bbN}{\mathbb{N}} 
 
\newcommand{\bbQ}{\mathbb{Q}} 
\newcommand{\bbR}{\mathbb{R}} 
\newcommand{\bbC}{\mathbb{C}} 
\newcommand{\bbK}{\mathbb{K}} 

\newcommand{\OqMmp}{\oc_q(\mc_{m,p}(}
\newcommand{\mmpc}{\OqMmp\bbK))}
\newcommand{\pmmpc}{\OqMmp\bbC))}

\newcommand{\Mmpc}{\mc_{m,p}(\bbC)}
\newcommand{\OMmpc}{\oc(\Mmpc)}
\newcommand{\Mmpnonneg}{\mc_{m,p}^{\ge0}(\bbR)}

\newcommand{\lc}{\mathcal{L}}
\newcommand{\wop}{w_\circ^p}
\newcommand{\wom}{w_\circ^m}
\newcommand{\woN}{w_\circ^N}
\newcommand{\Ecirc}{E^\circ}
\newcommand{\MmpK}{\mc_{m,p}(K)}
\newcommand{\onem}{\gc 1,m \dc}
\newcommand{\onep}{\gc 1,p \dc}

\newcommand{\onel}{\gc 1,l \dc}

\DeclareMathOperator{\chr}{char}
\DeclareMathOperator{\Spec}{Spec}
\newcommand{\HSpec}{\hc\mbox{-}\Spec}
\newcommand{\Honespec}{H_1\mbox{-}\Spec}
\newcommand{\Htwospec}{H_2\mbox{-}\Spec}
\DeclareMathOperator{\Fract}{Fract}

\input xy \xyoption{matrix} \xyoption{arrow}
\def\edge{\ar@{-}}
\def\plc{*{}+0}

\def\bigblacksquare{\vrule height 8pt depth 2pt width 10pt}
\def\mdblk{\save+<0ex,0ex>\drop{\bigblacksquare}\restore}

\input pstricks

\title{Torus-invariant prime ideals in quantum matrices, 
totally nonnegative cells and
symplectic leaves}

\author{K.R. Goodearl\thanks{\,The research of the first named author was supported
by a grant from the National Science Foundation (USA).},~~S. Launois\thanks{\,The research of the second named author was supported by a Marie Curie European Reintegration Grant within the
$7^{\mbox{th}}$ European Community Framework Programme.}~~and
T.H. Lenagan}

\long\def\symbolfootnote[#1]#2{\begingroup\def\thefootnote{\fnsymbol{footnote}}\footnote[#1]{#2}\endgroup}

\def\keywords#1{\def\@keywords{#1}}
\def\@keywords{}
\def\subjclass#1{\def\@subjclass{#1}}
\def\@subjclass{}

\date{}
\begin{document}
\maketitle

\subjclass{16W35; 15A48, 17B37, 17B63, 20G42, 53D17}
\symbolfootnote[0]{{\it 2000 Mathematics Subject Classification.}\enspace \@subjclass. }

\keywords{Quantum matrices; torus-invariant prime ideals; quantum minors; totally nonnegative cells; symplectic leaves}
\symbolfootnote[0]{{\it Key words and phrases.}\enspace \@keywords.}

\abstract{\footnotesize The algebra of quantum matrices of a given size supports a 
rational torus action by automorphisms. It follows from work of Letzter and the first named author that to understand the prime and primitive spectra of this algebra, the first step is to understand the prime ideals that are invariant under the torus action. In this paper, we prove that a family of quantum minors is the set of all quantum minors that belong to a given torus-invariant prime ideal of a quantum matrix algebra if and only if the corresponding family of minors defines a non-empty totally nonnegative cell in the space of totally nonnegative real matrices of the appropriate size. As a corollary, we obtain explicit generating sets of quantum minors for the torus-invariant prime ideals of quantum matrices in the case where the quantisation parameter $q$ is transcendental over $\bbQ$.}

\section*{Introduction} 

In recent publications, the same combinatorial
description has arisen for three separate objects of interest: torus-invariant prime ideals in quantum matrix algebras $\mmpc$ \cite{Cau2}, torus-orbits of symplectic leaves in matrix Poisson varieties $\Mmpc$ \cite{BGY}, and totally nonnegative cells in spaces $\Mmpnonneg$ of totally nonnegative matrices \cite{Pos}. Connections between the second and third of these objects were developed in \cite{GLL}. Here we close the circle by linking the first and second objects. More detail follows.

Many quantum algebras have a natural action by an algebraic torus, and a key ingredient in the study of the structure of these algebras is an understanding of the torus-invariant objects. For example, the Stratification Theory of Letzter and the first named author \cite{GoLet2} shows that, in the generic case, a complete understanding of the prime
spectrum of the quantised coordinate ring of $m\times p$ matrices, $\mmpc$, would start by classifying the (finitely many) prime ideals invariant under the natural action of the torus $\hc = (\bbK^\times)^{m+p}$. In \cite{Cau2},
Cauchon succeeded in counting the number of $\hc$-invariant prime ideals in
$\mmpc$. His method involved a bijection between certain diagrams,
now known as \emph{Cauchon diagrams}, and the $\hc$-invariant primes. Considerable progress in the understanding of quantum matrices has been made since that time by using Cauchon diagrams. 

The semiclassical limit of the quantum matrix algebras $\mmpc$ is the classical coordinate ring of the matrix variety $\mc_{m,p}(\bbK)$, endowed with a Poisson bracket that encodes the
nature of the quantum deformation which leads to quantum matrices. As a
result, the variety $\mc_{m,p}(\bbK)$ is endowed with a Poisson structure. In the complex case ($\bbK=\bbC$), a natural action of the torus $\hc = (\bbC^\times)^{m+p}$ leads to a stratification of the variety $\Mmpc$ via $\hc$-orbits of
symplectic leaves. In \cite{BGY}, Brown, Yakimov and the first named author showed that
there are finitely many such $\hc$-orbits of symplectic leaves. Each $\hc$-orbit is defined by certain rank conditions on submatrices. The classification
is given in terms of certain permutations from the relevant symmetric group
with restrictions arising from the Bruhat order. 

The totally nonnegative part of the space $\mc_{m,p}(\bbR)$ of real $m\times p$ matrices consists of those matrices
whose minors are all nonnegative. One can specify a cell decomposition of the
set $\Mmpnonneg$ of totally nonnegative matrices by specifying exactly which minors are to
be zero/non-zero. In \cite{Pos}, Postnikov classified the non-empty cells by
means of a bijection with certain diagrams, known as \emph{Le-diagrams}. 

The interesting observation from the point of view of this work is that in
each of the above three sets of results, the combinatorial objects that arise
turn out to be the same! The definitions of Cauchon diagrams and Le-diagrams
are the same, and the restricted permutations arising in the
Brown-Goodearl-Yakimov study can be seen to correspond to Cauchon/Le diagrams via
the notion of pipe dreams.

Once one is aware of these combinatorial connections, the suggestion arises that there should be a
connection between torus-invariant prime ideals, torus-orbits of symplectic
leaves and totally nonnegative cells. 

In \cite{GLL}, we study this connection, and prove that a family of minors
defines a non-empty totally nonnegative cell in the
space of totally nonnegative matrices if and only if this family is exactly the set of minors
that vanish on the closure of a certain torus-orbit of symplectic leaves in
the matrix Poisson variety. In the present note, we complete the picture and study the
quantum case. Our main result is the following.
\medskip

\noindent {\bf Theorem~\ref{maintheo}.}
{\em 
Let $\mathcal{F}$ be a family of minors in the coordinate ring of the affine variety $\Mmpc$, for some positive integers $m$, $p$, and let $\mathcal{F}_q$ be the corresponding family of quantum minors in $\mmpc$, where $\bbK$ is a field and $q\in \bbK^\times$ is a non-root of unity. Then the following are equivalent:
\begin{enumerate}
\item The totally nonnegative cell associated to $\mathcal{F}$ in $\Mmpnonneg$ is non-empty. 
\item $\mathcal{F}$ is the set of all minors that vanish on the 
closure of some $(\bbC^\times)^{m+p}$-orbit of symplectic leaves in $\Mmpc$.
\item $\mathcal{F}_q$ is the set of all quantum minors that belong to some 
$(\bbK^\times)^{m+p}$-invariant prime ideal in $\mmpc$.
\end{enumerate}}
\smallskip
\noindent (The torus actions in (2) and (3) are standard, and are  recalled below.)
\medskip

The proof of this result (see Section \ref{section:main result}) relies on two
algorithms, the \emph{deleting-derivations algorithm} and its inverse the \emph{restoration algorithm}, that were first developed
for use in quantum matrices \cite{Cau, Cau2, Lau2}. We note that recently and
independently Casteels \cite{Ca} has developed graph-theoretic methods (also
based on the restoration algorithm) to compute the set of quantum minors that
belong to a torus-invariant prime in $\mmpc$.

The sets of minors that vanish on the closure of a torus-orbit of symplectic
leaves in $\Mmpc$ have been explicitly described in \cite{GLL},
based on results of Fulton \cite{Ful} and Brown-Goodearl-Yakimov \cite{BGY}. Further, Yakimov proved that the above sets of minors generate the ideals of polynomial functions vanishing on closures of torus-orbits of symplectic leaves \cite{Yak}.
As a consequence of the above theorem, we obtain explicit descriptions of the sets of quantum minors that
belong to a torus-invariant prime in $\mmpc$ (see Theorem \ref{qminorsHprimes}).

The importance of understanding the sets $\mathcal{F}_q$ rests on the conjecture of the first and third named authors that, in the generic case ($q$ not a root of unity), all torus-invariant prime ideals in $\mmpc$ are generated by quantum minors \cite{GoLen1}. 
In \cite{Lau1}, the second named author proved this conjecture when the base field $\bbK$ is the field of complex numbers and
the quantisation parameter $q$ is transcendental over the rationals. We extend that result here to arbitrary base fields of characteristic zero (Theorem \ref{extHprimegens}). Consequently, in that case we deduce from the above results explicit
generating sets of quantum minors for the torus-invariant prime ideals of
$\mmpc$. A different approach to this result, applicable
to many quantized coordinate algebras, has been recently and independently
developed by Yakimov in \cite{Yak}. Explicit generating sets for torus-invariant prime ideals in general will, of course, also follow if and when the above conjecture is established.\\

Throughout this paper, we use the following conventions:

$\bullet$ $\bbN$ denotes the set of 
natural numbers, and we set $\bbC^\times:=\bbC\setminus \{0\}$.

$\bullet$ If $I$ is any non-empty finite subset of 
$\bbN$, then $|I|$ denotes its cardinality.

$\bullet$ $\bbK$ is a field, 
$\bbK^{\times}:=\bbK\setminus
\{0\}$
and $q\in \bbK^{\times}$ is not a root of unity. 

$\bullet$ $m$ and $p$ are two positive integers.

$\bullet$ If $k$ is a positive integer, then $S_k$ denotes the group of
permutations of $\gc 1,k \dc:=\{1, \cdots, k\}$.

$\bullet$ Let $K$ be a $\bbK$-algebra and
$M=(x_{i,\alpha})\in \mc_{m,p}(K)$. If $I \subseteq \onem$ and $\Lambda
\subseteq \onep$ with $|I|=|\Lambda |=t \geq 1$, then we denote by $[I |
\Lambda ]_q(M)$ the corresponding \emph{quantum minor} of $M$. This is the element of
$K$ defined by: 
$$[I | \Lambda ]_q(M)= [i_1, \dots, i_k|\alpha_1 , \dots , \alpha_k]_q :=\sum_{\sigma \in S_k} (-q)^{l(\sigma)}
x_{i_1, \alpha_{\sigma (1)}} \cdots x_{i_k, \alpha_{\sigma (k)}},$$ 
where
$I=\{i_1, \dots, i_k\}$, $\Lambda=\{\alpha_1 , \dots , \alpha_k\}$ and
$l(\sigma)$ denotes the length of the 
$k$-permutation $\sigma$. Also, it is convenient
to allow the empty minor: $[\emptyset|\emptyset]_q(M) := 1 \in K$.
Whenever we write a quantum minor in the form $[i_1, \dots, i_k|\alpha_1 , \dots , \alpha_k]_q$, we tacitly assume that the row
and column indices are listed in ascending order, that is, $i_1< \cdots< i_l$
and $\alpha_1< \cdots< \alpha_l$.

\section{$\hc$-prime ideals of $\mmpc$.}
\label{sectionQuantumMatrices}

\subsection{Quantum matrices.}

We denote by $R:=\mmpc$ the standard quantisation of the ring of
regular functions on $m \times p$ matrices with entries in $\bbK$; 
the algebra $R$ is
the $\bbK$-algebra generated by the $m \times p $ indeterminates
$X_{\ia}$, for 
$1 \leq i \leq m$ and $ 1 \leq \alpha \leq p$, subject to the
following relations:
 \[
\begin{array}{ll}
X_{i, \beta}X_{i, \alpha}=q^{-1} X_{i, \alpha}X_{i ,\beta},
& (\alpha < \beta); \\
X_{j, \alpha}X_{i, \alpha}=q^{-1}X_{i, \alpha}X_{j, \alpha},
& (i<j); \\
X_{j,\beta}X_{i, \alpha}=X_{i, \alpha}X_{j,\beta},
& (i <j,\;  \alpha > \beta); \\
X_{j,\beta}X_{i, \alpha}=X_{i, \alpha} X_{j,\beta}-(q-q^{-1})X_{i,\beta}X_{j,\alpha},
& (i<j,\;  \alpha <\beta). 
\end{array}
\]

It is well known that $R$ can be presented as an iterated Ore extension over
$\bbK$, with the generators $X_{\ia}$ adjoined in lexicographic order.
Thus, the ring $R$ is a noetherian domain; its skew-field of
fractions is denoted by $F$ or $\Fract R$. 
Moreover, since $q$ is not a root of unity, it follows from
\cite[Theorem 3.2]{GoLet1} that all prime ideals of $R$ are completely
prime.

The \emph{quantum minors} in $R$ are the quantum minors of the matrix $(X_{\ia}) \in \mc_{m,p}(R)$. To simplify the notation, we denote by $[I | \Lambda ]_q$ the quantum minor of $R$ associated to the row-index set $I$ and the column-index set $\Lambda$.

It is easy to check that the torus $\hc:=\left( \bbK^\times
\right)^{m+p}$ acts on $R$ by $\bbK$-algebra automorphisms via:
$$(a_1,\dots,a_m,b_1,\dots,b_p).X_{\ia} = a_i b_\alpha X_{\ia} \quad {\rm
for~all} \quad \: (\ia)\in \gc 1,m \dc \times \gc 1,p \dc.$$ 
We refer to this action as the \emph{standard action} of $\left( \bbK^\times
\right)^{m+p}$ on $\mmpc$.
Recall that an \emph{$\hc$-prime ideal} of $R$ is a proper $\hc$-invariant ideal $P$ such that whenever $P$ contains a product $IJ$ of two $\hc$-invariant ideals, it must contain either $I$ or $J$. As $q$ is not a root of unity, it follows from
\cite[5.7]{GoLet2} that there are only finitely many $\hc$-primes in $R$ and that every
$\hc$-prime is completely prime. Hence, the $\hc$-prime ideals of $R$ coincide with the $\hc$-invariant primes. We denote by
$\HSpec(R)$ the set of $\hc$-primes of $R$.

\subsection{$\hc$-primes and Cauchon diagrams.}
\label{sectionCauchondiagram}

In \cite{Cau2}, Cauchon showed that his theory of deleting-derivations can be applied to the iterated Ore extension $R$. As a consequence, he was able 
to parametrise the set $\hc$-$\mathrm{Spec}(R)$ in terms of combinatorial objects called {\it Cauchon diagrams}.

\begin{defi}\cite{Cau2} {\rm
An $m\times p$ \emph{Cauchon diagram} $C$ is simply an $m\times p$ grid consisting of $mp$ squares in which certain squares are coloured black.  We require that the collection of black squares have the following property: If a square is black, then either every square strictly to its left is black or every square strictly above it is black.

We denote by $\mathcal{C}_{m,p}$ the set of $m\times p$ Cauchon diagrams.
}
\end{defi}

Note that we will often identify an $m \times p $ Cauchon diagram with the set of coordinates of its black boxes. Indeed, if $C \in \mathcal{C}_{m,p}$ and  $(\ia)
 \in \gc 1, m\dc \times \gc 1,  p\dc$, we will say that $(\ia) \in C$ if the box in row $i$ and column $\alpha$ of $C$ is black.

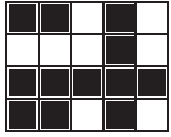
\begin{figure}[h]
\label{fig:CauchonDiagram}
\begin{center}
$$\xymatrixrowsep{0.01pc}\xymatrixcolsep{0.01pc}
\xymatrix{
\plc\edge[0,10]\edge[8,0] &&\plc\edge[8,0] &&\plc\edge[8,0]
&&\plc\edge[8,0] &&\plc\edge[8,0] &&\plc\edge[8,0] \\
 &\mdblk &&\mdblk &&&&\mdblk && &&\plc \\
\plc\edge[0,10] &&&&&&&&&&\plc \\
 &&&&&&&\mdblk && \\
\plc\edge[0,10] &&&&&&&&&&\plc \\
 &\mdblk &&\mdblk &&\mdblk &&\mdblk &&\mdblk \\
\plc\edge[0,10] &&&&&&&&&&\plc \\
 &\mdblk &&\mdblk &&&&\mdblk \\
\plc\edge[0,10] &&\plc &&\plc &&\plc &&\plc &&\plc 
}$$
\caption{An example of a $4\times 5$ Cauchon diagram}
\end{center}
\end{figure}

Recall \cite[Corollaire 3.2.1]{Cau2} that Cauchon has constructed (using the deleting-derivations algorithm) a bijection between $\hc$-$\mathrm{Spec}(\mmpc)$ and the collection $\mathcal{C}_{m,p}$. We discuss this bijection in more detail in Section \ref{section:main result}.

\begin{nota}
Let $C \in \mathcal{C}_{m,p}$. We denote by $J_C$ the unique $\hc$-prime ideal of $R$ corresponding to the Cauchon diagram $C$ under the above bijection.
\end{nota}

At this point, it is worth recalling that Cauchon diagrams are also closely
related to restricted permutations. More precisely, set 
\[
\mathcal{S}= S^{[-p,m]}_{m+p}:=\{w \in
S_{m+p} \ | \ -p \leq w(i) -i \leq m \mbox{ for all }i=1,2, \dots, m+p\}.
\]
The set 
$\mathcal{S}$ is an initial interval of the symmetric group $S_{m+p}$ endowed with the
Bruhat order. Namely, we have \cite[Proposition 1.3]{Lau3}, \cite[Lemma 3.12]{BGY}: 
\[
\mathcal{S}=\left\{ w
\in S_{m+p} \ \biggm| \ w \leq 
\begin{bmatrix} 1 & 2 & \dots & p &
p+1 & p+2 & \dots & m+p \\ 
m+1 & m+2 & \dots & m+p & 1 & 2& \dots & m
\end{bmatrix} \right\}.
\] 
It was proved in \cite[Corollary 1.5]{Lau3} that the
cardinality of $\mathcal{S}$ is equal to the number of $m \times p$ Cauchon diagrams. Note that one can construct an explicit bijection between these two sets by using the concept of pipe-dreams. (See \cite[Section 19]{Pos}.)

\subsection{Generators of $\hc$-prime ideals of $\mmpc$.}

In \cite{GoLen1}, the first and third named authors
conjectured that all
$\hc$-primes in $R$ are generated by quantum minors. (Of course, any prime generated by quantum minors is an $\hc$-prime, since every quantum minor of $R$ is an $\hc$-eigenvector.) They proved this
conjecture in the case where $m,p \leq 3$ \cite{GoLen1,GoLen2}. In
\cite[Th\'eor\`eme 3.7.2]{Lau1}, the second named author proved this conjecture in the case where
$\bbK=\bbC$ and the quantisation parameter $q$ is transcendental
over $\bbQ$. The result can then be extended to arbitrary base fields of characteristic $0$ (keeping $q$ transcendental over $\bbQ$), as we show below.

Note that the above conjecture is still open 
when we only assume that $q \in
\bbK^\times$ is not a root of unity.

Although in \cite{Lau2} an algorithm was developed 
that constructs, starting only from a
Cauchon diagram $C$, all of 
the quantum minors that belong to the $\hc$-prime ideal $J_C$, it is not easy to identify the families of
quantum minors that generate $\hc$-prime ideals. Casteels \cite{Ca} has
recently developed a graph theoretic method in order to compute these
families.  
  
In \cite[Theorem 4.2]{BGY}, Brown, Yakimov and the first named author
described the $\hc$-orbits of symplectic leaves of
$\Mmpc$ in terms of the vanishing and nonvanishing of
explicit sets of minors. Following the philosophy that symplectic leaves of
$\Mmpc$ should correspond bijectively to primitive
ideals of $\pmmpc$, they conjectured that a set of minors defines the closure
of an $\hc$-orbit of symplectic leaves if and only if the corresponding set of
quantum minors generates an $\hc$-prime ideal in $\pmmpc$ (see \cite[0.2]{BGY}).
This conjecture is proved here (see Theorem \ref{qminorsHprimes}).

\begin{lem} \label{H1H2prime}
Let $K_1\subseteq K_2$ be infinite fields, and let $A$ be a noetherian $K_2$-algebra supporting a rational action of a torus $H_2= (K_2^\times)^r$ by $K_2$-algebra automorphisms. Set $H_1 := (K_1^\times)^r$, which acts on $A$ by restriction of the $H_2$-action. Then the $H_1$-prime ideals of $A$ coincide with the $H_2$-prime ideals.
\end{lem}

\begin{proof} By \cite[Proposition II.2.9]{BrGo}, every $H_2$-prime of $A$ is prime, and consequently $H_1$-prime.

Rationality of the $H_2$-action on $A$ implies that $A$ has a $K_2$-basis $(a_i)_{i\in I}$ of $H_2$-eigenvectors whose $H_2$-eigenvalues are rational characters of $H_2$ \cite[Theorem II.2.7]{BrGo}. If $(\kappa_j)_{j\in J}$ is a basis for $K_2$ over $K_1$, then $(\kappa_j a_i)_{i\in I,j\in J}$ is a $K_1$-basis for $A$, consisting of $H_2$-eigenvectors with rational $H_2$-eigenvalues. Observe that the restriction to $H_1$ of any rational character of $H_2$ is a rational character of $H_1$. Consequently, each $\kappa_j a_i$ is an $H_1$-eigenvector whose $H_1$-eigenvalue is a rational character of $H_1$. Thus, the $H_1$-action on $A$ is rational.

Now any $H_1$-prime $P$ of $A$ is prime by \cite[Proposition II.2.9]{BrGo}. Consider the induced action of $H_2$ on $\Spec (A)$ by homeomorphisms, and let $S$ be the stabilizer of $P$ in $H_2$ under this action. It follows from \cite[Lemma II.2.8]{BrGo} that $S$ is a Zariski-closed subgroup of $H_2$. On the other hand, $S\supseteq H_1$, and $H_1$ is Zariski-dense in $H_2$ because $K_1$ is infinite. Therefore $S=H_2$, whence $P$ is an $H_2$-prime, as required.
\end{proof}

\begin{prop} \label{Hprimefunctoriality}
Let $K_1\subseteq K_2$ be fields, $q\in K_1^\times$ a non-root of unity, and identify $\OqMmp K_2))$ with $\OqMmp K_1)) \otimes_{K_1} K_2$. Set $H_i := (K_i^\times)^{m+p}$ for $i=1,2$, and let $H_i$ act on $\OqMmp K_i))$ by $K_i$-algebra automorphisms in the standard way. Then the rule $P\mapsto P\otimes_{K_1} K_2$ gives a bijection
$$\Honespec \OqMmp K_1)) \longrightarrow \Htwospec \OqMmp K_2)).$$
\end{prop}

\begin{proof} Set $A_i := \OqMmp K_i))$, and recall that $H_i$ acts rationally on $A_i$. By \cite[Theorem II.6.4]{BrGo}, every $H_i$-prime $P$ of $A_i$ is \emph{strongly $H_i$-rational} in the sense that $Z(\Fract A_i/P)^{H_i}= K_i$. (See the proof of \cite[Theorem II.5.14]{BrGo} for a verification of the required hypotheses.)

The action of $H_1$ on $A_1$ extends naturally to an action of $H_1 \equiv H_1\times \langle\text{id}\rangle$ on $A_2$ (by $K_2$-algebra automorphisms), and \cite[Proposition 3.3]{GoLen1} shows that the rule $P\mapsto P\otimes_{K_1} K_2$ gives a bijection $\Honespec (A_1) \rightarrow \Honespec (A_2)$. By Lemma \ref{H1H2prime}, $\Honespec (A_2)= \Htwospec (A_2)$, and the proposition is proved.
\end{proof}

\begin{theo} \label{extHprimegens}
Assume that $\chr(\bbK)=0$ and that $q$ is transcendental over $\bbQ$. Then every $\hc$-prime of $\mmpc$ is generated, as a left {\rm(}or right{\rm)} ideal, by the quantum minors that it contains. 
\end{theo}

\begin{proof} First, consider the subfield $K_1 := \bbQ(q) \subseteq \bbK$, and identify $\mmpc$ with $\OqMmp K_1))\otimes_{K_1} \bbK$. Set $H_1 := (K_1^\times)^{m+p}$, with the standard action on $\OqMmp K_1))$. By Proposition \ref{Hprimefunctoriality}, any $\hc$-prime of $\mmpc$ has the form $P\otimes_{K_1} \bbK$ for some $H_1$-prime $P$ of $\OqMmp K_1))$. Hence, it suffices to show that $P$ is generated, as a left (or right) ideal of $\OqMmp K_1))$, by the quantum minors it contains.

Now identify $K_1$ with a subfield of $\bbC$, and set $H_2 := (\bbC^\times)^{m+p}$, with the standard action on $\pmmpc$. By Proposition \ref{Hprimefunctoriality}, $P\otimes_{K_1} \bbC$ is an $H_2$-prime of $\pmmpc$, and thus by \cite[Th\'eor\`eme 3.7.2]{Lau1}, $P\otimes_{K_1} \bbC$ is generated, as a left (or right) ideal of $\pmmpc$, by the set $X$ of quantum minors it contains. Note that $X$ is also the set of quantum minors contained in $P$, and let $P'$ be the left ideal of $\OqMmp K_1))$ generated by $X$. Then $P'\otimes_{K_1} \bbC= P\otimes_{K_1} \bbC$, and consequently $P'=P$. Therefore $P$ is generated as a left ideal by $X$, and similarly as a right ideal.
\end{proof}

\section{$\hc$-orbits of symplectic leaves in $\Mmpc$.}

In this section, we study the standard Poisson structure of the
coordinate ring $\pmmpc$ coming from the commutators of $\pmmpc$. Recall that a
{\it Poisson algebra} (over $\bbC$) is a commutative
$\bbC$-algebra $A$ equipped with a Lie bracket $\{-,-\}$ which is a
derivation (for the associative multiplication) in each variable. The
derivations $\{a,-\}$ on $A$ are called {\it Hamiltonian derivations}. When
$A$ is the algebra of complex-valued $C^{\infty}$ functions on a smooth affine
variety $V$, one can use Hamiltonian derivations in order to define
Hamiltonian paths in $V$. A {\it Hamiltonian path in $V$} is a smooth path
$\gamma : [0,1] \rightarrow V$ such that there exists $f \in C^{\infty}(V)$
with $\frac{d\gamma}{dt}(t)=\xi_f(\gamma(t))$ for all $0 < t <1$, where
$\xi_f$ denotes the vector field associated to the Poisson derivation
$\{f,-\}$. It is easy to check that the relation ``connected by a piecewise
Hamiltonian path'' 
is an equivalence relation. The equivalence classes of this
relation are called the {\it symplectic leaves} of $V$; they form a partition
of $V$. 

A {\it Poisson ideal} of $A$ is any ideal $I$ such that $\{A,I\}
\subseteq I$, and a {\it Poisson prime} ideal is any prime ideal which
is also a Poisson ideal. The set of Poisson prime ideals in $A$ forms the {\it
Poisson prime spectrum}, denoted $\mathrm{PSpec} (A)$, which is given the relative Zariski
topology inherited from $\mathrm{Spec} (A)$.

\subsection{The Poisson algebra $\OMmpc$.}

Denote by $\OMmpc$ the coordinate ring of the variety
$\Mmpc$; note that 
$\OMmpc$ is a (commutative) polynomial algebra in
$mp$ indeterminates $Y_{\ia}$ with $1 \leq i \leq m$ and $1 \leq \alpha \leq
p$. 

The variety $\Mmpc$ is a Poisson variety: there is a unique Poisson bracket on the coordinate ring $\OMmpc $ determined by the following data. For all $(\ia) < (k,\gamma)$, we set: 
$$\{Y_{\ia} ,Y_{k,\gamma} \}=\left\{ \begin{array}{ll}
Y_{\ia}Y_{k,\gamma} &  \mbox{ if } i=k \mbox{ and } \alpha < \gamma \\
Y_{\ia} Y_{k,\gamma} &  \mbox{ if } i< k \mbox{ and } \alpha = \gamma \\
0 &   \mbox{ if } i < k \mbox{ and } \alpha > \gamma \\
2 Y_{i,\gamma} Y_{k,\alpha} &  \mbox{ if } i< k  \mbox{ and } \alpha < \gamma .  \\ 
\end{array} \right.$$
This is the standard Poisson structure on
the affine variety $ \Mmpc$ (cf.~\cite[\S1.5]{BGY}); the Poisson algebra structure on $\OMmpc$ is the semiclassical limit
of the noncommutative algebras $\pmmpc$. 

As with quantum minors in $\mmpc$, we abbreviate the notation for the minors of the matrix $Y= (Y_{\ia})$, writing $[I|\Lambda]:=[ I|\Lambda](Y)$.

Note that the Poisson bracket on $\OMmpc$ extends uniquely to a Poisson bracket on $\mathcal{C}^{\infty}(\Mmpc)$, so that $\Mmpc$ can be viewed as a Poisson manifold. Hence, $\Mmpc$ can be decomposed as the disjoint union of its symplectic leaves.

\subsection{Torus action.}

Notice that the torus $\hc:=\left( \bbC^\times \right)^{m+p}$ acts on
$\OMmpc$ by Poisson automorphisms via: $$(a_1,\dots,a_m,b_1,\dots,b_p).Y_{\ia}
= a_i b_\alpha Y_{\ia} \quad {\rm for~all} \quad \: (\ia)\in \gc 1,m \dc
\times \gc 1,p \dc.$$ 
We denote by $\hc$-$\mathrm{PSpec}(\OMmpc)$ the set of
those Poisson primes of $\OMmpc$ that are invariant under this action of
$\hc$. Note that $\hc$ is acting rationally on $\OMmpc$.

At the geometric level, this action of the algebraic torus $\hc$ comes from the left action of $\hc$ on $\Mmpc$ by Poisson isomorphisms via:
$$(a_1,\dots,a_m,b_1,\dots,b_p).M :=\mbox{diag}(a_1, \dots, a_m)^{-1} \cdot  M \cdot \mbox{diag}(b_1, \dots, b_p)^{-1}.$$
This action of $\hc$ on $\Mmpc$ induces an action of
$\hc$ on the set $\mathrm{Sympl}(\Mmpc)$ of symplectic
leaves in $\Mmpc$ (cf.~\cite[\S0.1]{BGY}). As in \cite{BGY}, we view the
$\hc$-orbit of a symplectic leaf $\lc$ as the set-theoretic union
$\bigcup_{h\in\hc} h.\lc \subseteq \Mmpc$, rather than as the family $\{h.\lc \mid
h\in\hc\}$. We denote the set of such orbits by
$\hc$-$\mathrm{Sympl}(\Mmpc)$. These orbits were described
by Brown, Yakimov and the first named author who obtained the following results.

We use the notation of \cite{BGY} except that we replace $n$ by $p$. 
In particular, we set $N = m+p$. Let $\wom$, $\wop$ and $\woN$ denote the
respective longest elements in $S_m$, $S_p$ and $S_N$, respectively, so that 
$w_0^r(i) = r + 1 - i$ for $i = 1, . . . , r$. 
Recall from equation (3.24) and Lemma 3.12 of
\cite{BGY} that 
\begin{equation} \label{woNS}
\woN \mathcal{S}=S_N^{\geq (\wop,\wom)} := \{w\in S_N\mid w\geq (\wop,\wom)\},
\end{equation} 
where
\[
(\wop,\wom) := \begin{bmatrix} 1 & 2 & \dots & p & p+1 & p+2 & \dots & p+m \\ 
p & p-1 & \dots & 1 & p+m & p+m-1& \dots & p+1\end{bmatrix} .
\]

\begin{theo} 
{\rm \cite[Theorems 3.9, 3.13, 4.2]{BGY}}
\begin{enumerate}
\item There are only finitely many $\hc$-orbits of symplectic leaves in
$\Mmpc$, and they are smooth irreducible locally
closed subvarieties. 
\item The set $\hc$-$\mathrm{Sympl}(\Mmpc)$ of orbits 
{\rm(}partially ordered by
inclusions of closures{\rm)} is isomorphic to the set $S_N^{\geq (\wop,\wom)}$ with
respect to the Bruhat order. 
\item Each $\hc$-orbit of symplectic leaves is defined by the vanishing and
nonvanishing of certain sets of minors. 
\item Each closure of an $\hc$-orbit of symplectic leaves is defined by the
vanishing of a certain set of minors.
\end{enumerate}
\end{theo}

For $y \in S_N^{\geq (\wop,\wom)}$, we denote by $\mathcal{P}_{y}$ the $\hc$-orbit
of symplectic leaves described in \cite[Theorem 3.9]{BGY}.

These results have several consequences for the (potential) link between the Poisson structure of $\OMmpc$ and the noncommutative structures of $\pmmpc$ and $\mmpc$. For instance, combined with \cite[Corollary 1.5]{Lau3}, the theorem shows that the number of $\hc$-orbits of symplectic leaves in  $\Mmpc$ is the same as the number of $\hc$-prime ideals in $\mmpc$, and so the same as the number of $m \times p$ Cauchon diagrams. 

In \cite[Theorem 2.9]{GLL}, we extended the results of the previous theorem. More precisely, we defined, for each restricted permutation $w \in \mathcal{S}$, a family $\mathcal{M}(w)$ of minors and proved that a minor $[I|\Lambda]$ vanishes on $\overline{\mathcal{P}_{\woN w}}$ 
if and only if $[I|\Lambda] \in \mathcal{M}(w)$. (The conditions used to define $\mathcal{M}(w)$ will be given below in Definition \ref{definition-M(w)}, to define the corresponding family $\mathcal{M}_q(w)$ of quantum minors.)

To finish this section, let us mention that the symplectic leaves in
$\Mmpc$ are algebraic; that is, they are locally
closed subvarieties of $\Mmpc$. As a consequence,
\cite[Proposition 4.8]{Goo} applies to our situation, so that there are only
finitely many Poisson $\hc$-primes in $\OMmpc$. More precisely, the Poisson
$\hc$-primes in $\OMmpc$ are in bijection with the $\hc$-orbits of symplectic
leaves in $\Mmpc$, and the minors that belong to a
Poisson $\hc$-prime ideal are exactly those that vanish on the closure of the
corresponding $\hc$-orbit of symplectic leaves in
$\Mmpc$. Hence, the number of Poisson
$\hc$-primes in $\OMmpc$ is the same as the number of $m \times p$ Cauchon
diagrams, and the families of minors that belong to Poisson $\hc$-primes are
exactly the families $\mathcal{M}( w)$ with $w \in \mathcal{S}$.

 In \cite[Section 5]{GLL}, we constructed an (explicit) bijection between the set of $m\times p$ Cauchon diagrams and the set of Poisson $\hc$-primes in $\OMmpc$. As in \cite[Theorem 5.3]{GLL}, we denote by $J'_C$ the unique Poisson $\hc$-prime associated to the Cauchon diagram $C$ under this bijection.

\section{$q$-quantum matrices and the deleting-derivations algorithm.}

For the remainder of this section, $K$ denotes a $\bbK$-algebra which is
also a skew-field. Except where otherwise stated, all the matrices that are
considered have their entries in $K$.

\subsection{$q$-quantum matrices.}

\begin{defi} (See \cite[Chapter 4]{PaWa}.)
Let $M=(x_{i,\alpha})\in \mc_{m,p}(K)$. We say that $M$ is a {\em $q$-quantum matrix} if the following relations hold between 
the entries of $M$:
$$
\begin{array}{ll}
x_{i, \beta}x_{i, \alpha}=q^{-1} x_{i, \alpha}x_{i ,\beta},
& (\alpha < \beta); \\
x_{j, \alpha}x_{i, \alpha}=q^{-1}x_{i, \alpha}x_{j, \alpha},
& (i<j); \\
x_{j,\beta}x_{i, \alpha}=x_{i, \alpha}x_{j,\beta},
& (i <j,\;  \alpha > \beta); \\
x_{j,\beta}x_{i, \alpha}=x_{i, \alpha} x_{j,\beta}-(q-q^{-1})x_{i,\beta}x_{j,\alpha},
& (i<j,\;  \alpha <\beta). 
\end{array}$$
\end{defi}

In order to define the deleting-derivations algorithm in the next
section, we will need the following notation.

\begin{nota}
\begin{itemize}
\item We denote by $\leq$ the lexicographic ordering on $\bbN^2$. Recall that
$$(\ia) \leq (j,\beta) \Longleftrightarrow 
[(i < j) \mbox{ or } (i=j \mbox{ and } \alpha \leq \beta )].$$
\item Set $\Ecirc= \left(\gc 1,m \dc \times \gc 1,p \dc \right)
\setminus \{(1,1)\}$ and $E= \Ecirc\cup \{(m,p+1)\}$.
\item Let $(j,\beta) \in \Ecirc$. 
Then $(j,\beta)^{+}$ denotes the smallest element 
(relative to $\leq$) of the set 
$\left\{ (\ia) \in E \ | \ (j,\beta) < (\ia) \right\}$.
\end{itemize}
 \end{nota}

The deleting derivations/restoration algorithms will be applied to 
matrices that are not necessarily $q$-quantum matrices. However, 
the matrices involved do have reasonable 
commutation relations that lead to the following definition.

\begin{defi} Let $M=(x_{i,\alpha}) \in \mc_{m,p}(K)$ and let $(j,\beta) \in E$. We say that $M$ is a {\em $(j,\beta)$-$q$-quantum matrix} 
if the following relations hold between the entries of $M$. If $\left( \begin{array}{cc} x & y \\ z & t \end{array} \right)$ is any $2 \times 2$ sub-matrix of $M$, then 
 \begin{enumerate}
 \item $ yx=q^{-1} xy, \quad zx=q^{-1} xz, \quad zy=yz, \quad ty=q^{-1}
 yt, \quad tz=q^{-1} zt.$
 \item If $t=x_{(k,\gamma)}$, then $tx= \begin{cases} xt &(\text{if\ } (k,\gamma)\ge(j,\beta))\\
 xt-(q-q^{-1})yz &(\text{if\ } (k,\gamma)<(j,\beta)). \end{cases}$
\end{enumerate}
\end{defi}

Note that our definitions of $q$-quantum matrix and $(j,\beta)$-$q$-quantum
matrix differ slightly from those of \cite[Definitions III.1.1 and
III.1.3]{Cau1}. Because of this, we must interchange $q$ and $q^{-1}$ whenever
carrying over results from \cite{Cau1}, and we need a shift of index in the above definition. More precisely, Cauchon's definition of a $(j,\beta)$-$q$-quantum matrix \cite[Definition III.1.3]{Cau1} matches, after switching $q$ and $q^{-1}$, our definition of a $(j,\beta)^+$-$q$-quantum matrix.

\subsection{The deleting-derivations algorithm.}

Now, we recall the deleting-derivations algorithm (see \cite{Cau2}, Convention
4.1.1 and \cite{Lau1}, Conventions 2.2.3). This algorithm plays a central role
in the study of the $\hc$-prime ideals of the algebra of generic quantum
matrices.

\begin{conv}[Deleting-derivations algorithm]
\label{conv1}
Let $M=(x_{\ia}) \in \MmpK$ be a matrix. As $r$ runs over the set $E$, we define matrices $M^{(r)} :=(x_{\ia}^{(r)}) \in \MmpK$ as follows.
\begin{enumerate}
\item \underline{When $r=(m,p+1)$}, the entries of the matrix $M^{(m,p+1)}$ are defined by 
$x_{\ia}^{(m,p+1)}:=x_{\ia}$ for all $(\ia) \in \gc 1,m \dc \times \gc 1,p \dc$. 
\item \underline{Assume that $r=(j,\beta) \in \Ecirc$} and that the matrix $M^{(r^{+})}=(x_{\ia}^{(r^{+})})$ has 
already been defined. The entries $x_{\ia}^{(r)}$ of the matrix $M^{(r)}$ are defined as follows.
\begin{enumerate}
\item If $x_{j,\beta}^{(r^+)}=0$, then $x_{\ia}^{(r)}=x_{\ia}^{(r^+)}$ 
for all $(\ia) \in \gc 1,m \dc \times \gc 1,p \dc$.
\item If $x_{j,\beta}^{(r^+)}\neq 0$ and $(\ia) \in \gc 1,m \dc \times \gc 1,p \dc$, then 

$x_{\ia}^{(r)}= \left\{ \begin{array}{ll}
x_{\ia}^{(r^+)}-x_{i,\beta}^{(r^+)} \left( x_{j,\beta}^{(r^+)}\right)^{-1} 
x_{j,\alpha}^{(r^+)}
& \qquad \mbox{if }i <j \mbox{ and } \alpha < \beta \\
x_{\ia}^{(r^+)} & \qquad \mbox{otherwise.} \end{array} \right.$ 
\end{enumerate}
We say that {\em $M^{(r)}$ is the matrix obtained from $M$ by applying the standard deleting-derivations 
algorithm at step $r$}, and $x_{j,\beta}^{(r^+)}$ is called the {\em pivot at step $r$}.
\item \underline{If $r=(1,2)$}, then we set $t_{\ia}:=x_{\ia}^{(1,2)}$ for all $(\ia) \in \gc 1,m \dc \times \gc 1,p \dc$. Observe that $x^{(r)}_{\ia}= x^{(r^+)}_{\ia}$ for all $r\le (\ia)$, and so $t_{\ia}= x^{(\ia)}_{\ia}= x^{(\ia)^+}_{\ia}$ for all $(\ia)\in \Ecirc$.
\end{enumerate}
\end{conv}

When $m=p$, step 2 of the deleting-derivations algorithm can be written as $M^{(j,\beta)}= \phi_{(m,j,\beta)}(M^{(j,\beta)^+})$ in the notation of \cite[\S III.2.3]{Cau1}.

\begin{lem}
\label{lemjbetasqquantique}
Let $(j,\beta) \in E$. If $M=(x_{i,\alpha})\in \MmpK$ is a $q$-quantum matrix, then the matrix $M^{(j,\beta)}$ 
is $(j,\beta)$-$q$-quantum.
\end{lem}

\begin{proof} When $m=p$, this follows from \cite[Proposition III.2.3.1]{Cau1} by induction on $(j,\beta)$. The rectangular case is proved in the same manner. It also follows by applying the square case to a square matrix obtained from $M$ by adjoining suitable zero rows or columns.
\end{proof}

\subsection{$\hc$-invariant $q$-quantum matrices.}
\label{subsection:qdet}

Before introducing the class of
$q$-quantum matrices that will be of interest for us, let us give some
notation for the quantum minors of a $q$-quantum matrix and the matrices
obtained by applying the deleting-derivations algorithm.

\begin{nota}
\label{notadet}
Let $M=(x_{\ia}) \in \MmpK$, and let $\delta= [I|\Lambda]_q(M)$ 
be a quantum minor of $M$. 
If $(j,\beta) \in E$, set
\[ 
\delta^{(j,\beta)}:= [I|\Lambda]_q(M^{(j,\beta)}).
\]
For $i\in I$ and $\alpha\in \Lambda$, set 
\[
\delta_{\widehat{i},\widehat{\alpha}}^{(j,\beta)}
:= [I\setminus\{i\} \mid \Lambda\setminus\{\alpha\}]_q(M^{(j,\beta)}),
\]
while 
\[
\delta_{\alpha \rightarrow \gamma}^{(j,\beta)}
:= [I \mid \Lambda\cup \{\gamma\} \setminus\{\alpha\}]_q(M^{(j,\beta)}) \quad (\gamma
\in \onep\setminus \Lambda)
\]
and 
\[
\delta_{i \rightarrow k}^{(j,\beta)}
:= [I\cup \{k\} \setminus\{i\} \mid \Lambda]_q(M^{(j,\beta)}) \quad (k \in \onem\setminus I).
\]
\end{nota}

In \cite{Lau1}, the effect of the deleting-derivations algorithm on quantum
minors was studied. Here we restrict our attention to a particular class of $q$-quantum
matrices.

\begin{defi}
Let $M=(x_{\ia})\in \MmpK$ be a $q$-quantum matrix. The matrix 
$M$ is said to be {\em $\hc$-invariant} if, for any $(j,\beta) \in \Ecirc$ and any quantum minor 
$\delta=[i_1,\dots,i_l |\alpha_1, \dots,\alpha_l]_q(M)$
 of $M$ such that $(i_l,\alpha_l) < (j,\beta)$, we have: 
$$ \delta ^{(j,\beta)^+}=0 \ \implies\ \delta^{(j,\beta)}=0.$$
\end{defi}

This definition is motivated by the fact proved in
\cite[Proposition 3.1.4.5]{Lau} that, if $J$ is an $\hc$-prime ideal
of the algebra $\mmpc$, then the matrix $(X_{i,\alpha} +J)$ (whose
entries are the canonical images of the generators of $\mmpc$ modulo the
$\hc$-prime $J$) is an $\hc$-invariant $q$-quantum matrix. For the convenience
of the reader, another proof of this result is presented in Corollary
\ref{H-invariant}.

 Also, in the case where $\bbK=\bbC$, it was proved in
 \cite[Theorem 5.4]{GLL} that, if $J'$ is a Poisson $\hc$-prime ideal of
 $\OMmpc$, then the matrix $(Y_{i,\alpha} +J')$ over $\OMmpc/J'$ is
 an $\hc$-invariant $1$-quantum matrix.

\subsection{Effect of the deleting-derivations algorithm on quantum minors.}

 The aim of this section is to obtain a characterisation of the quantum minors
 of $M^{(j,\beta)^+}$ that are equal to zero in terms of the quantum minors of
 $M^{(j,\beta)}$ that are equal to zero, for an $\hc$-invariant $q$-quantum matrix $M$. The first step does not require $\hc$-invariance.

\begin{prop}
\label{cor:criterion} 
Let $M=(x_{i,\alpha}) \in \MmpK$ be a $q$-quantum matrix and $r=(j,\beta) \in \Ecirc$. Set $u:=x_{j,\beta}^{(j,\beta)^+}$ and let $\delta=[i_1,\dots,i_l |\alpha_1, \dots,\alpha_l]_q(M)$  be a quantum minor of $M$ with $(i_l,\alpha_l) < (j,\beta)$. Assume that $u \neq 0$, $i_l <j$ and $\alpha_h < \beta < \alpha_{h+1}$ for some $h \in \gc 1 , l\dc$. {\rm(}By convention, $\alpha_{l+1}=p+1$.{\rm)} Then
\begin{eqnarray}
\label{formule4}
\delta^{(j,\beta)^+} u= \delta^{(j,\beta)} u
  +q \delta_{\alpha_h \rightarrow \beta}^{(j,\alpha_h)} 
 x_{j,\alpha_h}^{(j,\alpha_h)} .
  \end{eqnarray}
\end{prop}

\begin{proof} We proceed by induction on $l+1-h$. If $l+1-h=1$, then $h=l$ and $\alpha_l < \beta$. 
It follows from \cite[Proposition 2.2.8]{Lau1} that
$$\delta^{(j,\beta)^+}u = \delta^{(j,\beta)}u
  - \sum_{k=1}^l (-q)^{(l+1)-k}\delta_{i_k \rightarrow j}^{(j,\beta)} 
 x_{i_k,\beta}^{(j,\beta)} .$$
Moreover it follows from \cite[Proposition 2.2.8]{Lau1} that
$$\delta_{i_k \rightarrow j}^{(j,\beta)} = \delta_{i_k \rightarrow j}^{(j,\beta-1)} = \dots =\delta_{i_k \rightarrow j}^{(j,\alpha_l)}. $$
Then we deduce from \cite[Propositions 4.1.1 and 4.1.2]{Cau2} that 
$$\delta_{i_k \rightarrow j}^{(j,\beta)} =\delta_{\widehat{i_k}, \widehat{\alpha_l}}^{(j,\alpha_l)} x_{j,\alpha_l}^{(j,\alpha_l)}.$$
As $x_{i_k,\beta}^{(j,\beta)}= x_{i_k,\beta}^{(j,\beta-1)}= \dots = x_{i_k,\beta}^{(j,\alpha_l)}$ by construction, we obtain
$$\delta^{(j,\beta)^+} u= \delta^{(j,\beta)}u
  - \sum_{k=1}^l (-q)^{(l+1)-k}\delta_{\widehat{i_k}, \widehat{\alpha_l}}^{(j,\alpha_l)} x_{j,\alpha_l}^{(j,\alpha_l)}
 x_{i_k,\beta}^{(j,\alpha_l)} .$$
As $i_k < j$, $\alpha_l < \beta$ and the matrix $M^{(j,\alpha_l)}$ is $(j,\alpha_l)$-$q$-quantum, $x_{j,\alpha_l}^{(j,\alpha_l)}$ and $x_{i_k,\beta}^{(j,\alpha_l)}$ commute, so that
 $$\delta^{(j,\beta)^+}u = \delta^{(j,\beta)}u
  - \sum_{k=1}^l (-q)^{(l+1)-k}\delta_{\widehat{i_k}, \widehat{\alpha_l}}^{(j,\alpha_l)} x_{i_k,\beta}^{(j,\alpha_l)} x_{j,\alpha_l}^{(j,\alpha_l)}  .$$
Hence, by using a $q$-Laplace expansion \cite[Corollary A.5, Equation (A.6)]{GoLen1}, we obtain
$$\delta^{(j,\beta)^+}u = \delta^{(j,\beta)}u
  +q \delta_{\alpha_l \rightarrow \beta}^{(j,\alpha_l)} x_{j,\alpha_l}^{(j,\alpha_l)},$$
as desired.

Now let $l+1-h>1$, and assume the result holds for smaller values of $l+1-h$. Note that $\beta< \alpha_{h+1}\le \alpha_l$. Expand the quantum minor $\delta^{(j,\beta)^+} $ along its last column \cite[Corollary A.5, Equation (A.5)]{GoLen1}, to get
$$\delta^{(j,\beta)^+} = \sum_{k=1}^l (-q)^{k-l} x_{i_k,\alpha_l}^{(j,\beta)^+} \delta_{\widehat{i_k}, \widehat{\alpha_l}}^{(j,\beta)^+}.$$
The value corresponding to $l+1-h$ for the minors $\delta_{\widehat{i_k}, \widehat{\alpha_l}}^{(j,\beta)^+}$ is $l-h$, and so the induction hypothesis applies. We obtain
$$\delta_{\widehat{i_k}, \widehat{\alpha_l}}^{(j,\beta)^+} u = \delta_{\widehat{i_k}, \widehat{\alpha_l}}^{(j,\beta)} u+ q \delta^{(j,\alpha_h)}_{\substack{ \widehat{i_k}, \widehat{\alpha_l}\\ \alpha_h \rightarrow \beta }} x_{j,\alpha_h}^{(j,\alpha_h)}$$
for $k\in\onel$. As  
$x_{i_k,\alpha_l}^{(j,\beta)^+}
= x_{i_k,\alpha_l}^{(j,\beta)}= \dots 
= x_{i_k,\alpha_l}^{(j,\alpha_h)}$ by construction,  we obtain 
\begin{align*}
\delta^{(j,\beta)^+} u &= \sum_{k=1}^l (-q)^{k-l} x_{i_k,\alpha_l}^{(j,\beta)} \delta_{\widehat{i_k}, \widehat{\alpha_l}}^{(j,\beta)} u +q \sum_{k=1}^l (-q)^{k-l} x_{i_k,\alpha_l}^{(j,\alpha_h)}  \delta^{(j,\alpha_h)}_{\substack{ \widehat{i_k}, \widehat{\alpha_l}\\ \alpha_h \rightarrow \beta }} x_{j,\alpha_h}^{(j,\alpha_h)} \\
 &= \delta^{(j,\beta)} u
  + q \delta_{\alpha_h \rightarrow \beta}^{(j,\alpha_h)}
  x_{j,\alpha_h}^{(j,\alpha_h)} ,
\end{align*}
by two final $q$-Laplace expansions \cite[Corollary A.5, Equation (A.5)]{GoLen1}. This concludes the induction step.
\end{proof}

\begin{cor}
\label{H-invariant}
Let $J$ be an $\hc$-prime ideal of the algebra $R=\mmpc$. Then the matrix $(X_{i,\alpha} +J) \in \mc_{m,p}(R/J)$ is an $\hc$-invariant $q$-quantum  matrix.
\end{cor}

\begin{proof}
Set $x_{i,\alpha}=X_{i,\alpha} +J$ for all $(i,\alpha)$ and $M=(x_{i,\alpha}) \in \mathcal{M}_{m,p}(F_J)$, where $F_J$ denotes the skew-field of fractions of the noetherian domain $R/J$. Clearly, $M$ is a $q$-quantum matrix. 
Let $(j,\beta) \in \Ecirc$ and  let 
$\delta= [i_1,\dots,i_l|\alpha_1, \dots,\alpha_l]_q(M)$ be a quantum minor of $M$ such
that
$(i_l,\alpha_l) < (j,\beta)$. Assume that $ \delta ^{(j,\beta)^+}=0$. We need to
prove that $\delta^{(j,\beta)}=0$.  If $ \delta ^{(j,\beta)}= \delta
^{(j,\beta)^+}$,  then there is nothing to do; so  assume that 
$ \delta ^{(j,\beta)} \neq \delta ^{(j,\beta)^+}$. 
In this case, it follows from 
\cite[Proposition 2.2.8]{Lau1} that 
$u=t_{j,\beta}=x_{j,\beta}^{(j,\beta)^+} \neq 0$ and 
$i_l <j$ while 
$\alpha_h < \beta < \alpha_{h+1}$ for some $h \in \onel$. 
Thus,  Proposition \ref{cor:criterion} implies that  
$$
\delta^{(j,\beta)^+} u = \delta^{(j,\beta)} u 
  + q \delta_{\alpha_h \rightarrow \beta}^{(j,\alpha_h)}\, 
 x_{j,\alpha_h}^{(j,\alpha_h)} .
  $$
Hence,   
  $$
0= \delta^{(j,\beta)} u
  + q \delta_{\alpha_h \rightarrow \beta}^{(j,\alpha_h)} \,
 t_{j,\alpha_h} .
  $$

In order to conclude, observe that each 
$x_{\ia}^{(j,\beta)} =t_{\ia} +Q_{\ia}$
where the element $Q_{\ia}$ lies in the algebra
$$L := \bbK\langle t_{k,\gamma}^{\pm 1} \mid (k,\gamma) < (j,\beta) \text{\ and\ } t_{k,\gamma} \neq 0 \rangle,$$
and that each $x_{\ia}^{(j,\alpha_h)} =t_{\ia} +Q'_{\ia}$ 
where 
$$Q'_{\ia} \in \bbK\langle t_{k,\gamma}^{\pm 1} \mid (k,\gamma) < (j,\alpha_h) \text{\ and\ } t_{k,\gamma} \neq 0 \rangle \subseteq L.$$ 
(This can be proved by an easy induction, similar to 
\cite[Lemme 3.5.4]{Lau1}.)
Hence, $\delta^{(j,\beta)}$  and $\delta_{\alpha_h \rightarrow
\beta}^{(j,\alpha_h)}\, 
 t_{j,\alpha_h}$ both 
  belong to $L$. 
Finally, we note that \cite[Theorem 3.7]{Lau2}, which was proved for the case $\bbK=\bbC$, holds for the general coefficient field $\bbK$ (with the same proof). This result implies that the powers of $u=t_{j,\beta}$ are linearly independent over $L$. It follows that   
  $\delta^{(j,\beta)}=0$, as desired.
\end{proof}

We now restrict our attention to the case where $M$ is an $\hc$-invariant $q$-quantum matrix. In this case, we deduce from \cite[Proposition 2.2.8]{Lau1} and Proposition \ref{cor:criterion} the following characterisation of the quantum minors of $M^{(j,\beta)^+}$ that are equal to zero in terms of the quantum minors of $M^{(j,\beta)}$ that are equal to zero.

\begin{prop}
\label{prop:criterion}
Let $M=(x_{i,\alpha}) \in \MmpK$ be an $\hc$-invariant $q$-quantum matrix and $(j,\beta) \in \Ecirc$. Set $u:=x_{j,\beta}^{(j,\beta)^+}$. Let $\delta=[i_1,\dots,i_l | \alpha_1, \dots,\alpha_l]_q(M)$  be a quantum minor of $M$ with $(i_l,\alpha_l) < (j,\beta)$.

\begin{enumerate}
\item Assume that $u=0$. Then $\delta^{(j,\beta)^+}=0 $ if and only if 
$\delta^{(j,\beta)}=0$.
\item Assume that $u\neq 0$. If $i_l=j$, or if $\beta\in \{\alpha_1, \dots,\alpha_l \}$,
 or if $\beta < \alpha_1$, then $\delta^{(j,\beta)^+}=0$ if and only if 
 $\delta^{(j,\beta)}=0$.
\item Assume that 
$u \neq 0$, $i_l < j$ and $\alpha_{h} < \beta < \alpha_{h+1}$ 
for some $h \in \gc 1,  l\dc$.
Then $\delta^{(j,\beta)^+}=0$ if and only if $\delta^{(j,\beta)}=0$ and 
either $\delta_{\alpha_h \rightarrow \beta}^{(j,\alpha_h)} =0$ 
or $ x_{j,\alpha_h}^{(j,\alpha_h)}=0$.
\end{enumerate}
\end{prop}

\begin{proof} (1) and (2) follow from \cite[Proposition 2.2.8]{Lau1}. 
(3) follows from the previous proposition and the fact that $M$ is 
$\hc$-invariant.
\end{proof}


\section{Poisson $\hc$-primes of $\OMmpc$ versus $\hc$-primes in $\mmpc$.}
\label{section:main result}

We are now in position to study the quantum minors that belong to a 
given $\hc$-prime in the algebra $R= \mmpc$.

Associated to each $m\times p$ Cauchon diagram $C$ is an $\hc$-prime $J_C$ of $R$ determined as follows \cite{Cau2}. Via the deleting-derivations algorithm, $R$ is connected to a quantum affine space $\overline{R}$ with generators $T_{\ia}$ for $(\ia)\in \onem\times\onep$, and there is a canonical embedding $\varphi: \Spec(R) \rightarrow \Spec(\overline{R})$. The set $\{ T_{\ia} \mid (\ia)\in C \}$ generates an $\hc$-prime ideal of $\overline{R}$, and $J_C$ is the inverse image of this $\hc$-prime under $\varphi$. Moreover, the rule $C\mapsto J_C$ gives a bijection from $\mathcal{C}_{m,p}$ onto $\HSpec(R)$. (See \cite[Section 2]{Lau2} for a summary of the details.) It is the reverse connection -- going from $J_C$ back to $C$ -- that is important for the proof of the next theorem. To describe it, set $x_{i,\alpha}:= X_{i,\alpha}+J_C$ for $(\ia)\in \onem\times \onep$, consider the $q$-quantum matrix
$M_C:=(x_{i,\alpha}) \in \mc_{m,p}(R/J_C)$, and apply the deleting-derivations algorithm to $M_C$. For all $(\ia)$, (the $\bbK$-version of) \cite[Theorem 3.7]{Lau2} shows that
\begin{equation} \label{JCC}
x_{\ia}^{(1,2)}=0 \ \iff\ (\ia)\in C.
\end{equation}
Consequently, $x_{\ia}^{(\ia)}=0 \iff (\ia)\in C$ (recall that $x_{\ia}^{(1,2)}= x_{\ia}^{(\ia)}$).

Associated to $C$ is also a Poisson $\hc$-prime $J'_C$ of $\OMmpc$ which can be described in an analogous fashion. However, it is given in \cite[Section 5]{GLL} using the (Poisson) restoration algorithm (the inverse of the deleting-derivations algorithm), in the following way. First, let $A_C$ denote the polynomial ring
$$\bbC[ t_{\ia} \mid (\ia)\in (\onem\times\onep) \setminus C ],$$
and set $t_{\ia}=0 \in A_C$ for $(\ia)\in C$. A specific Poisson bracket is defined on $A_C$, 
and the restoration algorithm is applied to the matrix $(t_{\ia}) \in \mc_{m,p}(A_C)$. This leads to a Poisson algebra $A'_C := A_C^{(m,p+1)}$ with generators $y_{\ia}$, such that there is a Poisson algebra homomorphism $\varphi_C: \OMmpc \rightarrow A'_C$ sending $Y_{\ia} \mapsto y_{\ia}$ for all $(\ia)$. The Poisson $\hc$-prime ideal $J'_C$ associated to $C$ is defined as $J'_C := \ker(\varphi_C)$, and we identify $A'_C$ with $\OMmpc/J'_C$. The rule $C\mapsto J'_C$ defines a bijection from $\mathcal{C}_{m,p}$ onto $\hc$-$\mathrm{PSpec}(\OMmpc)$.

Application of the (Poisson) deleting-derivations algorithm \cite[Convention B.2]{GLL} to the matrix $M'_C := (y_{\ia}) \in \mc_{m,p}(A'_C)$ leads back to the matrix $(t_{\ia})$, so that $t_{\ia}= y_{\ia}^{(1,2)}$ for all $(\ia) \in \onem\times\onep$. In view of the definition of the $t_{\ia}$, we thus have
\begin{equation} \label{J'CC}
y_{\ia}^{(1,2)}=0 \ \iff\ (\ia)\in C.
\end{equation}
Consequently, just as with \eqref{JCC}, $y_{\ia}^{(\ia)}=0 \iff (\ia)\in C$.

\begin{theo}
\label{theo:comparison}
Let $C$ be an $m\times p$ Cauchon diagram. Let $J_C$ 
be the corresponding $\hc$-prime in $\mmpc$, and $J_C'$ the corresponding Poisson 
$\hc$-prime in $\OMmpc$.

Then, a quantum minor $[I|\Lambda]_q$ belongs to $J_C$ 
if and only if the corresponding minor $[I|\Lambda]$ belongs to $J_C'$.
\end{theo}

\begin{proof} Define the matrix
$M_C:=(x_{i,\alpha})$ as above, and observe that a quantum minor $\delta$ 
belongs to $J_C$ if and only 
if it corresponds to a quantum minor of $M_C$ that is equal to zero. 
Moreover, it follows from Corollary \ref{H-invariant} that 
$M_C$ is an $\hc$-invariant $q$-quantum matrix. 
In particular, Proposition \ref{prop:criterion}
applies to $M_C$.

Similarly, define the matrix $M'_C:=(y_{i,\alpha})$ as above, and observe
that a minor $\delta$ belongs to $J'_C$ if and only if it corresponds to a
minor of $M'_C$ that is equal to zero. Moreover,
\cite[Proposition 3.15]{GLL}  applies
to $M'_C$.

So it is sufficient to prove that a quantum minor of 
$M_C$ is equal to zero if and only if the corresponding minor 
of $M'_C$ is equal to zero. 

We prove, by induction on $(j,\beta) \in E$, that a quantum minor
$\delta^{(j,\beta)} :=[I|\Lambda]_q\left( x_{i,\alpha}^{(j,\beta)}
\right)$ of
$M_C^{(j,\beta)}$, where $I=\{i_1,\dots,i_l\}$ and 
$\Lambda=\{\alpha_1, \dots,\alpha_l\}$ with 
$(i_l,\alpha_l) < (j,\beta)$, 
is equal to zero if and
only if the corresponding minor 
$\Delta^{(j,\beta)} :=
[I|\Lambda]\left(
y_{i,\alpha}^{(j,\beta)} \right)$ of ${M'_C}^{(j,\beta)}$ is equal to zero. 

The case $(j,\beta)=(1,2)$ follows from the fact that we start with the same
Cauchon diagram. Then the induction step follows easily from the fact that one
can apply Proposition \ref{prop:criterion} to $M_C$ and \cite[Proposition
3.15]{GLL} to $M'_C$. 
Note that the induction step mimics the proof of
\cite[Theorem 3.16]{GLL}. However, for the convenience of the reader we
provide details.

We prove that 
$\delta^{(j,\beta)}=0 $ implies that $\Delta^{(j,\beta)}=0$. The reverse implication is proved similarly.

Assume first that  $(j,\beta)=(1,2)$. In this case, we have to prove that if 
$x_{1,1}^{(1,2)}=0$, then  $y_{1,1}^{(1,2)}=0$. 
Assume that  $x_{1,1}^{(1,2)}=0$. 
Then $(1,1) \in C$ by \eqref{JCC}, whence it follows from \eqref{J'CC} that 
$y_{1,1}^{(1,2)}=0$, as desired.

Now let $(j,\beta) \in \Ecirc$, 
and assume the result proved at step $(j,\beta)$. Suppose that
$\delta^{(j,\beta)^+}=0
$. In order to prove that  
$\Delta^{(j,\beta)^+}=0 $, we consider 
several cases. Keep in mind that $t_{j,\beta}= x^{(j,\beta)}_{j,\beta}= x^{(j,\beta)^+}_{j,\beta}$.

$\bullet$ Assume that $(i_l,\alpha_l) = (j,\beta)$. We distinguish between two cases.

$\bullet \bullet$ If  $t_{j,\beta}=x_{j,\beta}^{(j,\beta)}=0$, then $(j,\beta) \in C$ by \eqref{JCC}. Hence, it follows from \eqref{J'CC} that $y_{j,\beta}^{(j,\beta)}=0$, 
and so we deduce from \cite[Proposition 3.10]{GLL} that $\Delta^{(j,\beta)^+} =0$, as required.

$\bullet \bullet$ If  $t_{j,\beta}=x_{j,\beta}^{(j,\beta)} \neq 0$,
 then it follows from \cite[Proposition 4.1.2]{Cau2} that 
$0=\delta^{(j,\beta)^+} 
= \delta_{\widehat{j},\widehat{\beta}}^{(j,\beta)} x_{j,\beta}^{(j,\beta)}$, 
so that $\delta_{\widehat{j},\widehat{\beta}}^{(j,\beta)}=0$ . 
Hence, it follows from the induction hypothesis that 
$\Delta_{\widehat{j},\widehat{\beta}}^{(j,\beta)}=0$. 
As $\Delta^{(j,\beta)^+} 
= \Delta_{\widehat{j},\widehat{\beta}}^{(j,\beta)} y_{j,\beta}^{(j,\beta)}$, 
by \cite[Proposition 3.10]{GLL}, 
it follows that 
$\Delta^{(j,\beta)^+} =0$, as required.

$\bullet$ Assume that $(i_l,\alpha_l) < (j,\beta)$. 
We distinguish between three cases 
(corresponding to the three cases of Proposition 
\ref{prop:criterion}).

$\bullet \bullet$  Assume that $t_{j,\beta}= 0$. 
As we are assuming that $\delta^{(j,\beta)^+}=0 $, it follows from 
Proposition \ref{prop:criterion} 
that $\delta^{(j,\beta)}=\delta^{(j,\beta)^+}=0 $. 
Hence, we deduce from the induction hypothesis that  $\Delta^{(j,\beta)}=0$. 
On the other hand, as 
$t_{j,\beta} = 0$, we have $(j,\beta)\in C$ and so $y_{j,\beta}^{(j,\beta)}=0$.  Thus, it
follows from \cite[Proposition 3.15]{GLL} that 
$\Delta^{(j,\beta)^+}=\Delta^{(j,\beta)}=0 $, as desired.

$\bullet \bullet$  Assume that $t_{j,\beta}\neq 0$, 
and that $i_l=j$, or that $\beta\in \{\alpha_1,\dots,\alpha_l\}$,
or that $\beta < \alpha_1$. 
As we are assuming that $\delta^{(j,\beta)^+}=0 $, 
it follows from Proposition 
\ref{prop:criterion} that $\delta^{(j,\beta)}=\delta^{(j,\beta)^+}=0 $. 
Hence, we deduce from the induction hypothesis that  $\Delta^{(j,\beta)}=0$. 
On the other hand, as 
$t_{j,\beta} \neq  0$, we have $(j,\beta) \notin C$ and so $y_{j,\beta}^{(j,\beta)} \neq 0$. 
Moreover, as  
$i_l=j$,  or $\beta\in \{\alpha_1,\dots,\alpha_l\}$, 
or $\beta < \alpha_1$, 
it follows from \cite[Proposition 3.15]{GLL} that 
$\Delta^{(j,\beta)^+}=\Delta^{(j,\beta)}=0 $, as desired.

$\bullet \bullet$ Assume that 
$t_{j,\beta}\neq 0$ and $i_l 
< j$, while $\alpha_{h} < \beta < \alpha_{h+1}$ 
for some $h \in \onel$. 
Then, as in the previous case, we have $y_{j,\beta}^{(j,\beta)} \neq 0$. 
Moreover, it follows from Proposition \ref{prop:criterion} that 
$\delta^{(j,\beta)^+}=0 $ implies $\delta^{(j,\beta)}=0$ and 
either $\delta_{\alpha_h \rightarrow \beta}^{(j,\alpha_h)} =0$ 
or $ x_{j,\alpha_h}^{(j,\alpha_h)} =0$. 
Hence, we deduce from the induction hypothesis that $\Delta^{(j,\beta)}=0$ 
and 
either 
$\Delta_{\alpha_h \rightarrow \beta}^{(j,\alpha_h)} =0$ 
or $ y_{j,\alpha_h}^{(j,\alpha_h)} =0$. 
Finally, it follows from \cite[Proposition 3.15]{GLL} that 
$\Delta^{(j,\beta)^+}=0$, as desired. 
\end{proof}

\begin{theo}
\label{maintheo} Let $\mathcal{F}$ be a family of minors in the coordinate ring $\OMmpc$, and let $\mathcal{F}_q$ be the corresponding family of quantum minors in $\mmpc$. Then the following are equivalent:
\begin{enumerate}
\item The totally nonnegative cell associated to $\mathcal{F}$ in $\Mmpnonneg$ is 
non-empty. 
\item $\mathcal{F}$ is the set of all minors that vanish on the closure of some $\hc$-orbit of symplectic leaves in $\Mmpc$.
\item $\mathcal{F}_q$ is the set of all quantum minors that belong to some $\hc$-prime in $\mmpc$.
\end{enumerate}
\end{theo}

\begin{proof} 
The equivalence of (1) and (2) is proved in
\cite[Theorem 6.2]{GLL}. 

On the other hand, as discussed in \S2.2 above, the Poisson $\hc$-primes in $\OMmpc$ are in bijection with the
$\hc$-orbits of symplectic leaves in $\Mmpc$, and the
minors that belong to a Poisson $\hc$-prime ideal are exactly those that
vanish on the closure of the corresponding $\hc$-orbit of symplectic leaves in
$\Mmpc$. Thus, (2) holds if and only if $\mathcal{F}$ is the set of all minors
that belong to some Poisson $\hc$-prime ideal in $\OMmpc$. 
The remaining equivalence, (2)$\Longleftrightarrow$(3), now follows from Theorem \ref{theo:comparison}. 
\end{proof}

The families of minors that vanish on the closure of an $\hc$-orbit of
symplectic leaves have been explicitly described in \cite[Theorem
2.11]{GLL}. They are parametrised by the
set $\mathcal{S}$ of restricted permutations. Theorem \ref{maintheo} shows
that the quantum analogues of these families provide the families of quantum
minors that belong to $\hc$-primes in $\mmpc$.

Let us now be more precise. In \cite[Definition 2.6]{GLL} a family of minors
$\mathcal{M}(w)$ is defined for each $w \in \mathcal{S}$. The corresponding 
definition for families of quantum minors is given below. 

\begin{defi} \label{definition-M(w)}
For $w\in \mathcal{S}$, define $\mathcal{M}_q(w)$ to be the set of those quantum minors $[I|\Lambda]_q$ of $\mmpc$, with
$I\subseteq \onem$ and $\Lambda\subseteq \onep$, that satisfy at least one of the following conditions. 
\begin{enumerate}
\item $I\not\leq \wom w(L)$ for all $L\subseteq \onep\cap
w^{-1} \onem$ such that $|L|=|I|$ and $L\leq \Lambda$. 
\item $m+\Lambda\not\leq w\woN(L)$ for all $L\subseteq \onem\cap
\woN w^{-1} \gc m+1, N\dc$ such that $|L|=|\Lambda|$ and $L\leq I$. 
\item There exist $1\leq r \leq s\leq p$ such that 
$|\Lambda\cap \gc r, s \dc|> |\gc r, s \dc\setminus 
w^{-1} \gc m+r,\, m+s\dc |$.
\item There exist $1\leq r \leq s\leq m$ such that 
$|I\cap \gc r, s \dc|>|\woN \gc r, s \dc \setminus 
w^{-1} \wom \gc r, s \dc |$.
\end{enumerate}
\end{defi}

We are now in position to establish the following result which answers
positively a conjecture of Brown, Yakimov and the first named author in the
case where $\chr(\bbK)=0$ and $q$ is transcendental over
$\bbQ$.

\begin{theo} \label{qminorsHprimes}
\begin{enumerate}
\item Let $J$ be an $\hc$-prime ideal of $\mmpc$. Then there exists  $w \in \mathcal{S}$ such that $\mathcal{M}_q(w)$ is exactly the set of those quantum minors that belong to $J$.
\item Assume that $\chr(\bbK)=0$ and $q$ is transcendental over
$\bbQ$. Then we have:
$$\hc \mbox{-}\mathrm{Spec}(\mmpc)=\{ \langle \mathcal{M}_q(w) \rangle ~|~ w \in \mathcal{S} \}. $$
\end{enumerate}
\end{theo}

\begin{proof} 
(1) is a consequence of Theorem \ref{maintheo} 
and \cite[Theorem 2.11]{GLL}. Then (2) follows from (1) and Theorem \ref{extHprimegens}. \end{proof}

A different approach to the second part of this result, applicable to many
quantized coordinate algebras, has been recently and independently developed
by Yakimov in \cite{Yak} (see \cite[Theorem 5.5]{Yak}).

\section*{Acknowledgements}

The results in this paper were announced during the mini-workshop
``Non-neg\-a\-tiv\-i\-ty is a quantum phenomenon'' that took place at the
Mathematisches For\-schungs\-in\-sti\-tut Oberwolfach, 1--7 March 2009,
\cite{MFO}; we thank the director and staff of the MFO for providing
the ideal environment for this stimulating meeting. We also
thank Konni Rietsch, Laurent Rigal, Lauren Williams and Milen Yakimov for discussions and comments concerning this paper both at the workshop and at other times. Finally, we thank Karel Casteels for sending us his preprint \cite{Ca}.

\bibliographystyle{amsplain}
\bibliography{biblio}

\vskip 1cm

\providecommand{\bysame}{\leavevmode\hbox to3em{\hrulefill}\thinspace}
\providecommand{\href}[2]{#2}

\vskip 1cm

\noindent K.R. Goodearl:\\
Department of Mathematics,\\
 University of California,\\
 Santa Barbara, CA 93106, USA\\
Email: {\tt goodearl@math.ucsb.edu} \\

\noindent 
S. Launois: \\
School of Mathematics, Statistics and Actuarial Science,\\
University of Kent\\
Canterbury, Kent CT2 7NF, UK\\
Email: {\tt S.Launois@kent.ac.uk} \\

\noindent 
T.H. Lenagan: \\
Maxwell Institute for Mathematical Sciences\\
School of Mathematics, University of Edinburgh,\\
James Clerk Maxwell Building, King's Buildings, Mayfield Road,\\
Edinburgh EH9 3JZ, Scotland, UK\\
E-mail: {\tt tom@maths.ed.ac.uk}

 \end{document}